\documentclass[12pt]{amsart}

\usepackage[ansinew]{inputenc}
\usepackage[a4paper,dvips]{geometry}
\usepackage{latexsym}
\usepackage{amsfonts}
\usepackage{amsmath,amssymb}
\usepackage{bbm}
\usepackage{color}
\usepackage{bm}

\usepackage{graphicx}
\newtheorem{theorem}{Theorem}

\newtheorem{problem}{Problem}
\newtheorem{lemma}{Lemma}
\newtheorem{remark}{Remark}

\newtheorem{conjecture}{Conjecture}
\newtheorem{observation}{Observation}
\newtheorem{question}{Question}

\newcommand{\NN}{\mathbb{N}}

\newcommand{\ZZ}{\mathbb{Z}}
\newcommand{\RR}{\mathbb{R}}

\newcommand{\FF}{\mathbb{F}}
\newcommand{\F}{{\mathbb F}}
\newcommand{\Fp}{{\mathbb F}_p}

\newcommand{\Sy}{\mathfrak{S}}
\newcommand{\Sq}{{\mathcal S}}

\newcommand{\Fam}{\mathcal{F}}

\title[On the intriguing search for good permutations]{On the intriguing search for good permutations}
\author{Florian Pausinger}
\address{School of Mathematics \& Physics, Queen's University Belfast, BT7 1NN, Belfast, United Kingdom.}
\email{f.pausinger@qub.ac.uk}
\dedicatory{Dedicated to Henri Faure on the occasion of his 80th birthday}

\date{}

\begin{document}

\begin{abstract}
The intriguing search for permutations that generate generalised van der Corput sequences with exceptionally small discrepancy forms an important part of the research work of Henri Faure. On the occasion of Henri's 80th birthday we aim to survey (some of) his contributions over the last four decades which considerably improved our understanding of one-dimensional van der Corput sequences and inspired a lot of related work. We recall and compare the different approaches in the search for generalised van der Corput sequences with low discrepancy, i.e., using a single generating permutation versus using a sequence of permutations. Throughout, we collect, sharpen and extend open questions which all stem from the extensive work of Henri and his coworkers and which will hopefully inspire more work in the future.
\end{abstract}

\maketitle
\tableofcontents

\section*{Preface}

These notes form the basis for my presentation at MCQMC 2018 in Rennes, where I was invited to give a talk in the special session \emph{Low discrepancy sequences and point sets} which was dedicated to Henri Faure on the occasion of his 80th birthday and organised by Friedrich Pillichshammer (JKU Linz) and Wolfgang Schmid (University of Salzburg). The main aim of this paper is to pick one particular aspect of the research of Henri Faure, i.e., the search for permutations that generate (generalised) van der Corput sequences with exceptionally small discrepancy, which led to a series of interesting results and still leaves intriguing open questions. For a comprehensive account on the van der Corput sequence, its various generalisations and different discrepancy estimates I would like to refer to the recent survey of Faure, Kritzer, and Pillichshammer \cite{survey1}.

What intrigues me and motivates this article is the fact that despite the considerable amount of excellent work collected in \cite{survey1} it is still not entirely clear which structure underlies those one-dimensional sequences with to-date smallest known discrepancy. The systematic search for \emph{good} permutations which was initiated by Henri Faure gave us sequences with exceptionally small discrepancy. However, attempts to describe and reproduce the structure of such permutations have, so far, not been entirely satisfying. 
The search for an answer to a seemingly hopeless problem usually inspires interesting and original work.
The aim of this paper is to present the history as well as the state-of-the-art in the search for one-dimensional infinite sequences with low discrepancy. Moreover, we collect new numerical results and various open problems hoping to motivate future research.

\section{Extending a classical definition}


\subsection{A bit of context and history}
The theory of uniform distribution modulo one (abbreviated as u.d. mod 1) can be traced back to the work of Tchebycheff and Kronecker in the late 19th century and of Hardy and Littlewood in the beginning of the 20th century. In a seminal paper of 1916, Weyl \cite{we16} placed the concept of u.d. mod 1 in the center of the study of Diophantine approximation in one and higher dimensions and systematised all the earlier work on the subject. We recall the basic definitions in the following and use this opportunity to introduce the notation we use throughout.

Let $X=(x_i)_{i\geq1}$ be a sequence of real numbers. We denote the fractional part of a real number $x$  by $\{ x\}$. We put 
$\mathrm{I}=[0, 1)$ and consider  
$[\alpha,\beta) \subseteq \mathrm{I}$.
For $N\geq 1$, let $A([\alpha,\beta),N,X)$ denote the number of indices $1 \leq i\leq N$ for which $\{x_i\} \in [\alpha,\beta)$.
An infinite sequence $X$ is \emph{uniformly distributed modulo 1} (u.d. mod 1) if
\begin{equation} \label{udt1}
\underset{N\rightarrow \infty} \lim \frac{A([\alpha,\beta), N, X) }{N} = \beta - \alpha, 
\end{equation}
for every subinterval $[\alpha,\beta)$ of $\mathrm{I}$. The uniform distribution property \eqref{udt1} of an infinite sequence is usually quantified with one of several different notions of discrepancy. Following Faure, i.e. not dividing by $N$, we put 
$$E([\alpha,\beta), N, X)=A([\alpha,\beta),N,X)- (\beta-\alpha) N,$$
and define
the star discrepancy, $D_N^*(X)$, and the extreme discrepancy, $D_N(X)$, of the first $N$ points of $X$ as
$$D_N^*(X)=\underset{[0,\alpha) \subseteq \mathrm{I}} \sup | E([0,\alpha), N, X) |, \ \ \ \text{ and } \ \ \ D_N(X)=\underset{[\alpha,\beta) \subseteq \mathrm{I}} \sup | E([\alpha,\beta), N, X) |.$$
For completeness, we also define the $L^2$-discrepancy, $D^{(2)}_N(X)$, and the diaphony, $F_N(X)$, of the first $N$ points of $X$ as
\begin{align*}
D^{(2)}_N(X)&:= \left( \int_0^1 \vert E([0,\alpha), N, X) \vert ^2 d\alpha \right)^{1/2},\\
F_N(X)&:=\left(2\pi^2 \int_0^1 \int_0^1 \vert E([\alpha,\beta), N, X) \vert ^2 d\alpha d\beta \right)^{1/2}.
\end{align*}
This definition of diaphony is taken from \cite{fa05}; for the original definition in terms of exponential sums we refer to \cite{zi76}.

There are two classical examples of uniformly distributed sequences. The fact that the so-called Kronecker sequence $(\{n\alpha\})_{n \geq 1}$ is u.d. mod 1 for every irrational $\alpha$ was established in 1909-10 independently by Bohr, Sierpinski and Weyl. 
Ostrowski \cite{Ost22} showed in 1922, that a Kronecker sequence has very small discrepancy if the irrational $\alpha$ has bounded partial quotients.
In 1935, the van der Corput sequence, $\mathcal{S}_2$, -- we recall its definition in the next section -- was introduced and shown to be u.d. mod 1. 
In fact, it was already known to van der Corput that $D_N(\mathcal{S}_2) = O(\log N)$; \cite{vdc35a, vdc35b}. Similarly, it was already known in the 1920s that the optimal order of the discrepancy of Kronecker sequences is $\log N$ as well; \cite{Beh24, Ost22}. However, it was only shown in 1972 that this order of magnitude is best possible when Schmidt \cite{sch72} improved a general lower bound result of Roth \cite{roth54} for the special case of one-dimensional sequences. He proved that for any infinite sequence $X \subset \mathrm{I}$ we have that
\begin{equation} \label{schmidt}
D_N(X) \geq c \cdot \log N,
\end{equation}
for infinitely many integers $N$, where $c>0$ is an absolute constant. These early results set the stage for this paper.

\subsection{Chasing the constants} 
In 1972 the question of the optimal order of magnitude that can be achieved is finally settled and there are two different constructions available. Which one is better? And what is the largest constant $c$ for which \eqref{schmidt} remains true? To formalise this question we define the asymptotic constants $s(X)$ and $s^*(X)$ of an infinite sequence $X$ as
\begin{equation*}
s(X) = \limsup_{N \rightarrow \infty} \frac{D_N(X) }{\log N}, \ \ \ \text{ and } \ \ \ s^*(X) = \limsup_{N \rightarrow \infty} \frac{D^*_N(X) }{\log N},
\end{equation*}
which are known to be finite for our particular examples above. In general, whenever $s(X)$ is finite for a sequence $X$, we say that this sequence has \emph{low discrepancy}. Since we know examples of low discrepancy sequences, it makes sense to define the one-dimensional discrepancy constant $\hat{s}$ as
$$ \hat{s} = \inf_X \ \limsup_{N \rightarrow \infty} \frac{D_N(X)}{\log N} = \inf_X s(X), $$
in which the infimum is taken over all infinite sequences $X$; we define $\hat{s}^*$ in a similar fashion. In other words, $\hat{s}$ is the supremum over all $c$ such that \eqref{schmidt} holds for all $X$ for infinitely many $N$.
Larcher \cite{Lar15, Lar16} and Larcher \& Puchhammer \cite{LarPu16} recently improved earlier results of B\'{e}jian \cite{Bej82} and proved that
$$ 0.12112 < \hat{s}, \ \ \ \text{ and } \ \ \ 0.06566 < \hat{s}^*; $$
for the most recent upper bounds we refer to Table~\ref{table:discr}.

Returning to our discussion of the history, apart from a result of Haber \cite{hab66} on the constant $s(\mathcal{S}_2)$ nobody had calculated such constants until then. There were only upper bounds on $s(X)$ and $s^*(X)$ for different sequences. It lasted until the late 1970s before Ramshaw \cite{ram81}, B\'{e}jian \& Faure \cite{bej77} and S\'{o}s \& Dupain \cite{dup84} were able to give exact formulas for the asymptotic constants of various sequences; see Table~\ref{table1}.

\begin{table}[h]
\begin{center}
\begin{tabular}{|c|c|c|}
\hline
$X$ & $s(X)$ & $s^*(X)$\\
\hline
&&\\
$\mathcal{S}_2$ & $\frac{1}{3 \log 2} = 0.4808\ldots$ &  $\frac{1}{3 \log 2} = 0.4808\ldots$\\[3pt]
$(\{n\alpha \})$ \text{ for } $\alpha= (\sqrt{5}-1)/2$ & $\frac{1}{5 \log(1/\alpha)} = 0.4156\ldots $ &  $\frac{3}{20 \log(1/\alpha)} = 0.3117\ldots $\\[3pt]
$(\{n\alpha \})$ \text{ for } $\alpha= \sqrt{2}-1$ & $\frac{1}{2 \log(1/\alpha)} = 0.5672\ldots $ &  $\frac{1}{4 \log(1/\alpha)} = 0.2836\ldots $ \\
&&\\
\hline
\end{tabular}
\end{center}
\caption{Sequences $X$ with smallest known values for $s^*(X)$ and $s(X)$ before the paper of Henri Faure in 1981.}
\label{table1}
\end{table}
At this point it seemed that the $(\{n\alpha \})$ sequence was outperforming the van der Corput sequence.
The breakthrough -- in case you are a fan of the van der Corput sequence -- came in 1981 when Faure published his influential paper \cite{fa81}. 
This paper contains six theorems, all of which became of great importance for the subsequent development.
However, the first big contribution of this paper was the \emph{definition} of a new family of sequences, which Faure called \emph{generalised van der Corput sequences}.

\subsection{Extending a definition}
We start with the definition of the classical van der Corput sequence $\mathcal{S}_2$. 
Let $a_j(n)$ denote the $j$th coefficient in the $b$-adic expansion
$$n=a_0 + a_1 b + a_2 b^2 + \ldots = \sum_{j=1}^{\infty} a_j b^j$$ of an integer $n$, where $0 \leq a_j(n) \leq b-1$ and if $0\leq n < b^k$ then $a_j(n)=0$ for $j\geq k$.
The $b$-adic radical inverse function is defined as $S_b: \NN_0 \rightarrow [0,1)$,
$$ S_b(n) = \frac{a_0(n)}{b} + \frac{a_1(n)}{b^2} + \frac{a_2(n)}{b^3} + \ldots = \sum_{j=0}^{\infty} \frac{a_j(n)}{b^{j+1}},$$
for $n \in \NN_0$. The classical van der Corput sequence is defined as $\mathcal{S}_2=(S_2(n))_{n\geq 0}$.

Furthermore, let $\Sy_b$ be the set of all permutations of $\{0,1,\ldots, b-1\}$, i.e., all permutations of the first $b$ non-negative integers, and let $\Sigma=(\sigma_j)_{j\geq 0}$ be a sequence of permutations in base $b$ such that $\sigma_j \in \Sy_b$ for all $j\geq0$. Throughout this paper we sometimes write a concrete permutation $\sigma$ as tuple, e.g. $\sigma=(0,2,4,1,3)$ meaning that $0$ is mapped to $0$, $1$ is mapped to $2$, $2$ is mapped to $4$ and so on; i.e. the number at the $i$-th position of the tuple is the image of $i$ under $\sigma$. Moreover, $\sigma \circ \tau$ denotes the usual composition of permutations for $\sigma, \tau \in \Sy_b$.
Following Faure \cite{fa81}, we define the generalised (or permuted) van der Corput sequence $\mathcal{S}_b^{\Sigma}=(S_b^{\Sigma}(n))_{n\geq 0}$ for a fixed base $b\geq 2$ and a sequences of permutations $\Sigma$ by
\begin{equation*}
S_b^{\Sigma}(n)=\sum_{j=0}^{\infty} \frac{\sigma_j( a_j(n) )}{b^{j+1} }.
\end{equation*}
Every such sequence is uniformly distributed; \cite[Propri\'et\'e 3.1.1]{fa81}. The most general definition was considered in \cite{cf93} where $\Sigma$ was even allowed to be a sequence of permutations in variable bases.
If we consider a constant sequence of permutations, i.e., we use the same permutation $\sigma$ for every $j$, then we write $\mathcal{S}_b^{\sigma}$. Hence, $\mathcal{S}_2=\mathcal{S}_2^{id}$, where $id$ denotes the identity permutation in $\Sy_b$. This extension of the basic definition opens the door to a huge space of possibilities and it suggests two different approaches when trying to improve the constants in Table~\ref{table1}. We can either try to find a single permutation that generates a sequence with small asymptotic constant (see Section~\ref{sec3} and Section~\ref{sec:families}) or we can look for a sequence of permutations (see Section~\ref{sec:sequences}). 

\subsection{An influential paper} \label{sec1:results}
But first, let's return to our discussion of the influential paper \cite{fa81}.
The first three theorems of the paper establish a method to calculate the discrepancy as well as the asymptotic constant of generalised van der Corput sequences exactly -- we will come back to these results in more detail in Section~\ref{sec2}. Using this method, Faure was able to explicitly calculate these values in three different settings, giving the final three theorems of the paper. He improved the so far best upper bounds on $\hat{s}$ and $\hat{s}^*$.
In particular, he shows in \cite[Th\'{e}or\`{e}me 4]{fa81} for $b=12$ and $\sigma=(0,7,3,10,5,2,9,6,1,8,4,11)$ that
$$ 0.375 < s(\mathcal{S}_{12}^{\sigma}) < 0.38.$$
In \cite[Th\'{e}or\`{e}me 5]{fa81} another permutation in base $b=12$ is considered. Let\\ $\sigma=(0,5,9,3,7,1,10,4,8,2,6,11)$ and define the set $A\subset \NN$ as
$$ A=\bigcup_{H=1}^{\infty} A_H, \ \ \ \text{ with } \ \ \ A_H=\{ H(H-1)+1, H(H-1)+2, \ldots, H^2\}. $$
Moreover, let $\tau_b \in \Sy_{b}$ with $\tau_b(k)=b-k-1$ for $0 \leq k \leq b-1$ and define $\Sigma_A=(\sigma_j)_{j\geq0}$ with $\sigma_j=\sigma$ if $j\in A$ and $\sigma_j = \tau_{12} \circ \sigma$ if $j \notin A$. Then
$$s^*(\mathcal{S}_{12}^{\Sigma_A^{\sigma}}) = \frac{1919}{3454 \log 12} = 0.22358\ldots. $$
We will come back to this particular and important construction of a sequence of permutations out of a single permutation later in this paper.
Finally, in  \cite[Th\'{e}or\`{e}me 6]{fa81} the constants $s$ and $s^*$ for van der Corput sequences that are generated from the identity permutation in base $b$ are calculated:
$$ s(\mathcal{S}_b^{id}) = s^*(\mathcal{S}_b^{id}) = \left\{
\begin{array}{ll}
\frac{b-1}{4 \log b} & \mbox{ if $b$ is odd,}\\
\frac{b^2}{4 (b+1) \log b}& \mbox{ if $b$ is even. }
\end{array} \right.$$
Consequently, the only sequence generated from an identity permutation that improves the result of the classical sequence in base $2$ is $\mathcal{S}_3^{id}$. However, $s(S_3^{id})=0.4551\ldots$ is still larger than the value Ramshaw obtained for $(\{ n \alpha \})$. We will further explore the connection between generalised van der Corput sequences and $(\{ n \alpha \})$ sequences in Section~\ref{sec:families}.
Importantly, it is also shown \cite[Corollaire 3]{fa81} that the sequence generated from the identity has the largest asymptotic constant among all permuted van der Corput sequences in a fixed base $b$. In other words, we always have for fixed $b$ and $\sigma \in \Sy_b$ that
\begin{equation} \label{id}
s(\mathcal{S}_b^{\sigma}) \leq s(\mathcal{S}_b^{id}).
\end{equation}

The results and methods that Faure developed in his paper of 1981 were the starting point for the intriguing search for good permutations.
\begin{question}
Which characteristics of a permutation or a sequence of permutations lead to a sequence with very small discrepancy? Is there a systematic way to construct/find a sequence with a very small asymptotic constant?
\end{question}

\subsection{Asymptotic constants and diaphony}
The upper bounds on the asymptotic discrepancy constants $\hat{s}$ and $\hat{s}^*$ that Faure obtained in 1981 could be improved over the years. The currently best values are collected in Table~\ref{table:discr}; see the appendix for the explicit permutations. 

\begin{table}[h]
\begin{center}
\begin{tabular}{|c|l|l|l|}
\hline
Faure & 1981, \cite{fa81} & $\hat{s} < 0.38$ &$\hat{s}^* < 0.22358$ \\
Faure & 1992, \cite{fa92} & $\hat{s} < 0.367$ &  \\
Ostromoukhov & 2009, \cite{os09} & $\hat{s} < 0.354$ & $\hat{s}^* < 0.222223$ \\
\hline
\end{tabular}
\end{center}
\caption{Chronological order of improvements of upper bounds on $\hat{s}$ and $\hat{s}^*$ after 1981.}
\label{table:discr}
\end{table}

Chaix \& Faure \cite{cf93} extended the methods of \cite{fa81} to the study of diaphony (as well as $L^2$-discrepancy); see \cite[Section~4.1]{cf93}. It turns out that the asymptotic constant
$$ f(X) = \limsup_{N \rightarrow \infty} \frac{F^2_N(X) }{\log N}
$$
can be calculated with a method that is similar to the case of (star) discrepancy; see \cite[Section~4.2]{cf93}. 
In particular, this result was used to improve the until then best known bound for the diaphony of infinite one-dimensional sequences (see \cite[Section~4.3]{cf93}) and raised the same questions as in the case of discrepancy. Is there a permutation that can further improve the smallest known value? How can we find it?
Setting
$$\hat{f} := \inf_X f(X)$$
Proinov \cite{pro83} provides a first upper bound on $\hat{f}$ in 1983 using the sequence $(\{n\alpha\})$, in which $\alpha$ is again the conjugate golden ratio. The currently smallest known upper bound for $(\{n\alpha\})$ is due to Xiao \cite{xiao90} who improved the value of Proinov in 1990 to $f(\{n\alpha\}) < 7.5$. We recall the improvements of the upper bound on $\hat{f}$ in chronological order in Table~\ref{table:diaph}; see the appendix for the explicit permutations. These bounds are all obtained from (generalised) van der Corput sequences.
Concerning lower bounds on $\hat{f}$, Proinov \cite{pro86} showed in 1986 that
$$ \hat{f} > 1/4624= 0.00021\ldots .$$

\begin{table}[h]
\begin{center}
\begin{tabular}{|c|l|l|}
\hline
Proinov \& Grozdanov  & $1988$, \cite{progro88} & $\hat{f}<14.238$ \\
Proinov \& Atanassov & $1988$, \cite{proat88} & $\hat{f}< 1.583$\\
Chaix \& Faure & $1993$, \cite{cf93} & $\hat{f}< 1.316$\\
Pausinger \& Schmid & $2010$, \cite{ps10} & $\hat{f}<1.137$\\
\hline
\end{tabular}
\end{center}
\caption{Improvements of the upper bound on $\hat{f}$ in chronological order.}
\label{table:diaph}
\end{table}

\begin{problem}
It is not satisfying that the upper bounds for $\hat{f}$ obtained from concrete sequences are much larger than the corresponding bounds for $\hat{s}$ while the so far best general lower bound for $\hat{f}$ is much smaller than the one for $\hat{s}$. Can the general lower bound of Proinov for $\hat{f}$ be improved?
\end{problem}

\paragraph{\bf Outline.} In Section~\ref{sec2} we introduce the method to calculate the discrepancy of generalised van der Corput sequences followed by Section~\ref{sec3} on how to find and construct good generating permutations. Section~\ref{sec:families} introduces permutation polynomials as a way to describe the structure of generating permutations, before we turn to sequences of generating permutations in Section \ref{sec:sequences}. Section~\ref{sec:context} gives a wider context to the results of Faure and, finally, Section~\ref{sec:conclusions} concludes the paper.

\section{Calculating the discrepancy}
\label{sec2}
In the following we introduce Faure's system of basic functions which can be used to calculate the discrepancy of generalised van der Corput sequences. Furthermore, we survey various powerful techniques for the analysis of structurally similar permutations. In the final part of this section we recall how to calculate the asymptotic discrepancy constant for a given sequence.

\subsection{Basic functions and exact formulas}
\label{sec:Disc}
The exact formulas for the discrepancy (and also the diaphony) of generalised van der Corput sequences are based on a set of elementary functions which are defined for any permutation $\sigma \in \Sy_b$. Let 
$$\mathcal{X}_b^{\sigma}:= \left( \frac{\sigma(0)}{b}, \frac{\sigma(1)}{b},\ldots ,\frac{\sigma(b-1)}{b} \right).$$ 
For $h \in \{0,1,\ldots, b-1\}$ and $x \in \left[\frac{k-1}{b},\frac{k}{b}\right[$, in which $1\leq k \leq b$ is an integer, we define $$\varphi_{b,h}^{\sigma}(x):=\left\{
\begin{array}{ll}
A([0,h/b[,k,\mathcal{X}_b^{\sigma})-h x & \mbox{ if } 0 \le h \le \sigma(k-1),\\[5pt]
(b-h)x- A([h/b,1[,k,\mathcal{X}_b^{\sigma}) & \mbox{ if }\sigma(k-1)< h < b.
\end{array}\right.$$
The functions $\varphi_{b,h}^{\sigma}$ are piecewise affine and are extended to the real numbers by periodicity.
To simplify the formulas for (star) discrepancy it is convenient to put
\begin{align*}
\psi_b^{\sigma,+} = \max_{0 \leq h < b} \varphi_{b,h}^{\sigma}, \ \ \ \ \psi_b^{\sigma,-} = \max_{0 \leq h < b} (- \varphi_{b,h}^{\sigma}), \ \ \ \
\psi_b^{\sigma}=\psi_b^{\sigma,+}+\psi_b^{\sigma,-}.
\end{align*}
For an infinite, one-dimensional sequence $X$ we set
$$ D_N^+ = \sup_{0 \leq \alpha \leq1} E( [0,\alpha], N, X ), \ \ \ \text{ and } \ \ \  D_N^- = \sup_{0 \leq \alpha \leq1} (- E( [0,\alpha], N, X )).$$
Then we get from \cite[Th\'{e}or\`{e}me 1]{fa81} for $N\geq 1$ that
\begin{align*} \label{formula:DN}
D_N^+(\mathcal{S}_{b}^{\Sigma}) = \sum_{j=1}^{\infty} \psi_b^{\sigma_{j-1},+} (N/b^j), \ \ \ \  & D_N^-(\mathcal{S}_{b}^{\Sigma}) = \sum_{j=1}^{\infty} \psi_b^{\sigma_{j-1},-} (N/b^j), \\
D_N(\mathcal{S}_{b}^{\Sigma}) = \sum_{j=1}^{\infty} \psi_b^{\sigma_{j-1}} (N/b^j), \ \ \ \
& D_N^*(\mathcal{S}_{b}^{\Sigma}) = \max \left(D_N^+(\mathcal{S}_{b}^{\Sigma}), D_N^-(\mathcal{S}_{b}^{\Sigma})  \right).
\end{align*}
Note that the infinite series in these formulas can indeed be computed exactly; for details we refer to \cite[Section~3.3.6, Corollaire 1]{fa81}. Similar formulas hold for diaphony as well as $L^2$-discrepancy \cite{cf93}.

\begin{figure}[hbt]
 \centering
 \includegraphics[width=.45\textwidth]{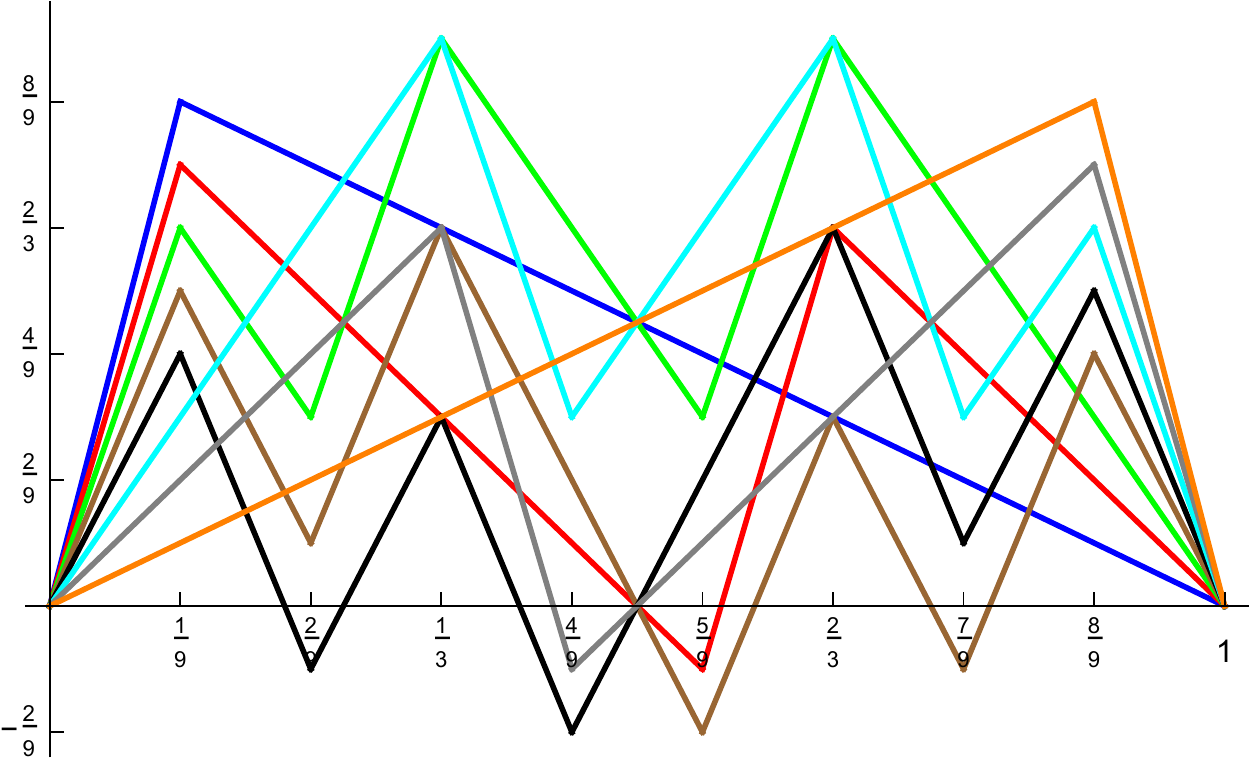} \,\,\,\,
\includegraphics[width=.45\textwidth]{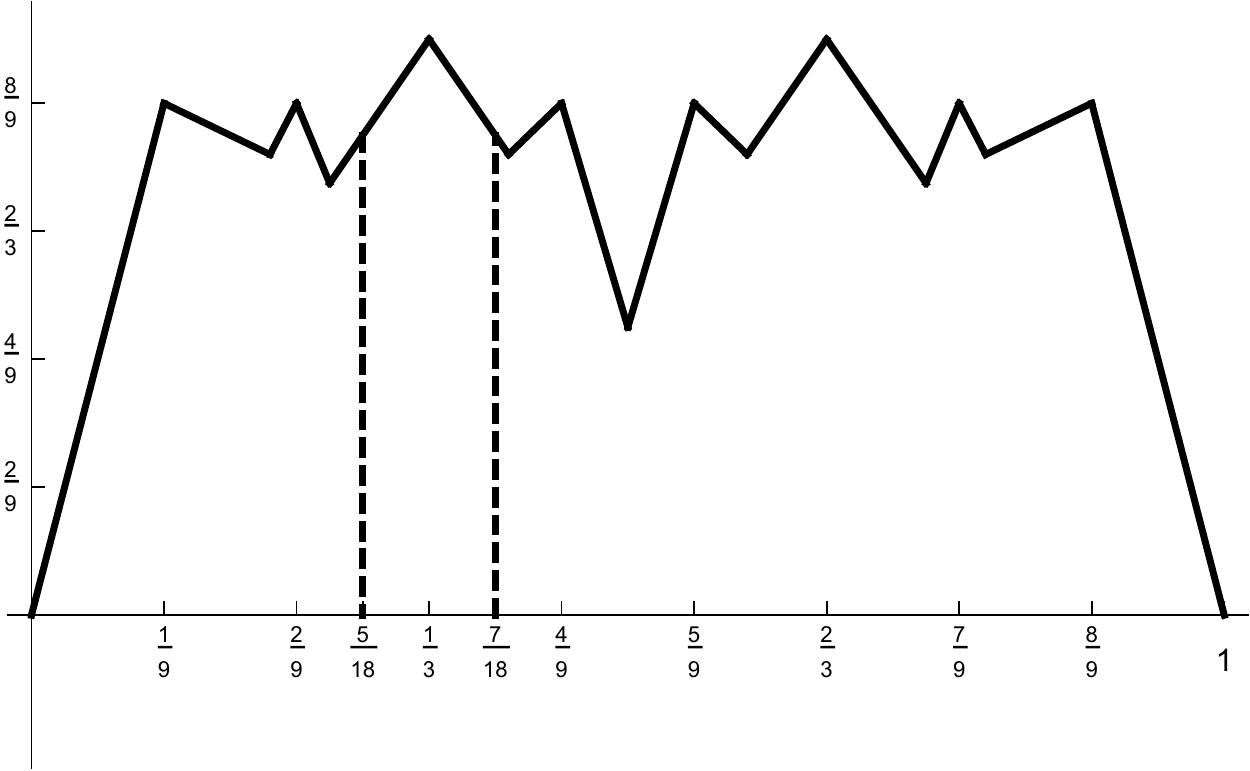}
\caption{
The functions $\varphi_{9,h}^{\omega}(x)$ and $\psi_9^{\omega}(x)$ for the permutation $\omega$ as defined in Section~\ref{sec3}. }
\label{fig:psi}
\end{figure}

\subsection{Symmetries, swapping and intrication of two permutations}
There are several important lemmas linked to the basic functions which facilitate the analysis of structurally similar permutations as well as sequences of permutations.
The subsequent Symmetry Lemma states that a shift or a reflection of a permutation does not change the discrepancy of the generated sequence which is to be expected since we interpret $[0,1)$ as a circle.

\begin{lemma}[Symmetry Lemma]  \label{lem:SameDisc}
Let $0<a<b$ be an integer, let $\sigma \in \Sy_b$ and let $\sigma', \sigma'' \in \Sy_b$ be defined as
$$ \sigma'(x) = \sigma(x) + a \pmod{b}, \ \ \ \text{ and } \ \ \  \sigma''(x) = - \sigma(x) \pmod{b}$$
for $0\leq x \leq b-1$. Then it holds for all $N$
\begin{equation*} 
D_N(\mathcal{S}_b^{\sigma}) = D_N(\mathcal{S}_b^{\sigma'}) \ \ \ \text{ and } \ \ \ D_N(\mathcal{S}_b^{\sigma}) = D_N(\mathcal{S}_b^{\sigma''}).
\end{equation*}
\end{lemma}
The first part of the lemma was noted in \cite[Th\'eor\`eme 4.4]{cf93}; the second assertion was discussed in \cite[Lemma~2.1]{patop18}.

Next is the Swapping Lemma which plays a crucial role in the asymptotic analysis of the star discrepancy of generalised van Corput sequences as well as two-dimensional Hammersley point sets (as discussed in Section~\ref{sec:context}) and first appeared as \cite[Lemme 4.4.1]{fa81}. Recall that we already defined $\tau_b \in \Sy_{b}$ as the permutation $\tau_b(k)=b-k-1$ for $0 \leq k \leq b-1$. 
The \emph{swapping} permutation $\tau_b$ gets its name from the fact that
\begin{align*}
\psi_{b}^{\sigma \circ \tau, +} &= \psi_b^{\sigma,-} \\
\psi_{b}^{\sigma \circ \tau, -} &= \psi_b^{\sigma,+}.
\end{align*}
This property is especially useful if $\sigma=id$ because then 
$$\psi_b^{\tau,+}=\psi_b^{id,-}=0 \ \ \ \text{ and } \ \ \ \psi_b^{\tau,-}=\psi_b^{id,+}=\psi_b^{id}.$$ 
In general, we have:
\begin{lemma}[Swapping Lemma] \label{lem:Swap}
Let $\mathcal{N} \subset \NN_0$ and let $\sigma \in \Sy_b$. Further define $\Sigma_{\mathcal{N}}^{\sigma} = (\sigma_j)_{j \geq 0}$ as $\sigma_j = \sigma$ if $j \in \mathcal{N}$ and $\sigma_j = \tau_b \circ \sigma$ if $j \notin \mathcal{N}$. Then
\begin{align*}
D_N^+(\mathcal{S}_b^{\Sigma_{\mathcal{N}}^{\sigma}} ) &= \sum_{j \in \mathcal{N}} \psi_b^{\sigma,+} (N/b^j) + \sum_{j \notin \mathcal{N}} \psi_b^{\sigma,-} (N/b^j), \\
D_N^-(\mathcal{S}_b^{\Sigma_{\mathcal{N}}^{\sigma}} ) &= \sum_{j \in \mathcal{N}} \psi_b^{\sigma,-} (N/b^j) + \sum_{j \notin \mathcal{N}} \psi_b^{\sigma,+} (N/b^j), \\
D_N^*(\mathcal{S}_b^{\Sigma_{\mathcal{N}}^{\sigma}} ) &= \max (D_N^+(\mathcal{S}_b^{\Sigma_{\mathcal{N}}^{\sigma}} ), D_N^-(\mathcal{S}_b^{\Sigma_{\mathcal{N}}^{\sigma}} )), \\
D_N(\mathcal{S}_b^{\Sigma_{\mathcal{N}}^{\sigma}} ) &= D_N(\mathcal{S}_b^{\sigma} ).
\end{align*}
\end{lemma}

Finally, Faure defined an operation \cite[Section~3.4.3]{fa81} which takes two arbitrary permutations $\sigma, \tau$ in bases $b$ and $c$ and outputs a new permutation, $\sigma \cdot \tau$ in base $b\cdot c$. The motivation for this definition comes from the following property which was first noted in \cite[Proposition 3.4.3]{fa81}.

\begin{lemma}[Intrication] \label{intr}
For $\sigma \in \Sy_b$ and $\tau \in \Sy_c$ define $\sigma \cdot \tau \in \Sy_{bc}$ as
$$ \sigma \cdot \tau (k'' b+ k' ) = c \, \sigma(k') + \tau(k''),$$
for $0\leq k' < b$ and $0 \leq k'' < c$. Then,
\begin{equation*} 
\psi_{bc}^{\sigma \cdot \tau}(x) = \psi_{b}^{\sigma}(c x) + \psi_{c}^{\tau}(x),
\end{equation*}
such that 
$$\max_{x \in \RR} \ \psi_{bc}^{\sigma \cdot \tau}(x) \leq \max_{x \in \RR} \ \psi_{b}^{\sigma}(x) + \max_{x \in \RR} \ \psi_{c}^{\tau}(x).$$
\end{lemma}

Note that if we set $\sigma=\tau$, then the intrication $\sigma \cdot \sigma$ gives a permutation in base $b^2$ whose $\psi$-function is the function $F_2^{\sigma}$ defined in \eqref{Fn} below. 
In this special case the new permutation generates the same sequence as the original permutation.

\subsection{Asymptotic analysis} \label{sec:asymp}
The exact formulas can be used for the asymptotic analysis of the discrepancy of generalised van der Corput sequences. By \cite[Th\'{e}or\`{e}me 2]{fa81}, 
\begin{equation} \label{asym}
s(\mathcal{S}_b^{\sigma}) =\limsup_{N \rightarrow \infty} \frac{D_N(\mathcal{S}_b^{\sigma}) }{\log N} = \frac{\alpha_b^{\sigma}}{ \log b} \ \ \ \text{ with } \ \ \ 
\alpha_b^{\sigma} = \underset{n\geq 1} \inf \ \underset{x \in \RR}  \sup \left( \frac{1}{n} \sum_{j=1}^n \psi_b^{\sigma} (x/b^j) \right).
\end{equation}
Furthermore, we introduce the function
\begin{equation} \label{Fn}
F_n^{\sigma}(x)= \sum_{j=0}^{n-1} \psi_b^{\sigma} (x b^j),
\end{equation}
and rewrite \eqref{asym} as
$$ \alpha_b^{\sigma} = \inf_{n\geq 1}  \left( \max_{x \in [0,1]} F_n^{\sigma}(x)/n  \right). $$
Note that the local maxima of $\psi_b^{\sigma}$ have arguments of the form $x=k/b$ for $k\in \NN$ with $0\leq k \leq b-1$. As shown in \cite[Lemme 4.2.2]{fa81}, the sequence $(\max_{x \in [0,1]} F_n^{\sigma}(x)/n)_{n\geq 1}$ is decreasing. In particular,
$$\alpha_b^{\sigma} \leq \ldots \leq \ \max_{x \in [0,1]} F_2^{\sigma}(x)/2 \ \leq \  \max_{x \in [0,1]} F_1^{\sigma}(x) =  \max_{x \in \RR} \ \psi_b^{\sigma},$$
with $\alpha_b^{\sigma} = \lim_{n \rightarrow \infty} \max_{x \in [0,1]} F_n^{\sigma}(x)/n$.

Finally, using the swapping permutation $\tau_b$ as well as the set $A\subset \NN$ with
$$ A=\bigcup_{H=1}^{\infty} A_H, \ \ \ \text{ with } \ \ \ A_H=\{ H(H-1)+1, H(H-1)+2, \ldots, H^2\}, $$
Faure proves in \cite[Th\'{e}or\`{e}me 3]{fa81}
$$ s^*(\mathcal{S}_b^{\Sigma_A^{\sigma}}) = \limsup_{N \rightarrow \infty} \frac{D_N^*(\mathcal{S}_b^{\Sigma_A^{\sigma}}) }{\log N} = \frac{\alpha_b^{\sigma,+} + \alpha_b^{\sigma,-}}{2 \log b}, $$
where
\begin{align*}
\alpha_b^{\sigma,+} = \inf_{n\geq 1} \sup_{x \in \RR}  \left( \frac{1}{n} \sum_{j=1}^n \psi_b^{\sigma,+} (x/b^j) \right), \ \ 
\alpha_b^{\sigma,-} = \inf_{n\geq 1} \sup_{x \in \RR}  \left( \frac{1}{n} \sum_{j=1}^n \psi_b^{\sigma,-} (x/b^j) \right).
\end{align*}

\section{Good permutations and how to find them}
\label{sec3}

\subsection{Searching for good permutations}
A closer look at the permutations listed in the appendix reveals that the currently best values of Ostromoukhov are obtained for a permutation in base $84$ resp. in base $60$ whereas Pausinger \& Schmid improved the value for the diaphony with a permutation in base $57$. 
Given the number of permutations up to base $84$, this seems much worse than looking for a needle in a haystack. 
So what is the secret behind the successful search for good permutations? The immediate answer appears disappointing at first.
The authors of both papers used a \emph{clever brute force} approach; i.e., testing permutations in a given base in a systematic way. However, recalling the large number of permutations, isn't it fascinating that such a strategy can work? What is hidden behind \emph{clever}?
The secret of the search algorithm(s) can be found in the definition of the basic $\varphi_{b,h}^{\sigma}$ functions:
\begin{observation} \label{obs1}
It is enough to know the set of the first $k$ values of $\sigma$, i.e. 
$$V_{k}^{\sigma}:=\{\sigma(i): 0\leq i \leq k-1\},$$ 
to calculate $\varphi_{b,h}^{\sigma}(x)$ for  $x \in [(k-1)/b,k/b[$ and $h \in \{0,1,\ldots ,b-1\}$. 
Hence, the value at $\psi_b^{\sigma}(k/b)$ for a given set $V_{k}^{\sigma}$ is the same for all permutations with this initial set of elements.
\end{observation}

This observation suggests the following strategy.
We can construct \emph{good} permutations step by step -- Ostromoukhov speaks of building a pruned tree \cite[Section~5]{os09}. Knowing the first $k$ images of a permutation, allows to calculate $\psi_b^{\sigma}(x)$ for $x \in [0,k/b[$.
Thus, we choose a pruning threshold $T$ and set without loss of generality $\sigma(0)=0$. This leaves $b-1$ possibilities for $\sigma(1)$. We can calculate the discrepancy value for each of these possibilities. If the discrepancy is bigger than the pruning value $T$ then the corresponding branch of the tree is pruned away. Therefore, we discard all permutations with this set of first $k$ images. Now, we continue to $\sigma(2)$ and so on. This procedure finds all permuations $\sigma$ with 
$$ \max_{x \in [0,1[} \psi_b^{\sigma}(x) < T.$$
The right choice of the pruning parameter $T$ is the artistic part of the game. If $T$ is too large, the final tree may contain a huge number of branches. If the value is too small, the final tree may contain no branches at all.
This first list of good permutations is now the starting point for a more careful analysis. 
In a second step, we can sort (and shorten) our list of permutations by calculating $\max_{x} F_2^{\sigma}(x)$. This leaves us (hopefully) with a small enough set of permutations for further resp. asymptotic analysis.

\subsection{Constructing good permutations}
The approach chosen in the previous section led to an improvement of the upper bounds on the asymptotic constants $\hat{s}, \hat{s}^*$ and $\hat{f}$ which is of theoretical interest. In a more practical context one often needs a large number of good generating permutations, i.e., in simulations or in numerical integration. However, the systematic search for permutations tells us only little about the actual structure of good permutations and how to explicitly construct them. The aim of this and the next section is to review results that reveal more about the structure of good permutations. 

As we have seen in Section~\ref{sec1:results}, the asymptotic constants $s(\mathcal{S}_b^{id})$ increase with increasing base. In \cite[Theorem~1.1]{fa92}, Faure contrasts this result by constructing a permutation $\omega=\omega_b$ in every base $b$ with an asymptotic constant that is independent of the particular base: 
\begin{theorem} [Faure \cite{fa92}] \label{thm92}
For every $b$ there exists a permutation $\omega$ such that $s(\Sq_b^{\omega}) \leq 1/ \log 2$.
\end{theorem}
The proof of this result is constructive; i.e., Faure provides an algorithm to explicitly construct the permutations $\omega_b$ from permutations in smaller bases. Start with $b=2$ and $\omega_{2}=id_2=(0,1)$.
Now suppose that all permutations $\omega_b$ have already been constructed for bases $b<b'$. If $b'=2b$ is even, then $\omega_{b'}= id_2 \cdot \omega_b$. If $b'=2b+1$ is odd, then we define $\omega_{b'}$ for $0\leq k<b $ as
\begin{align*}
\omega_{2b+1}(k)= \left\{
\begin{array}{ll}
\omega_{2b}(k), 	&\textrm{\(0\leq \omega_{2b}(k)<b\)}\\
\omega_{2b}(k)+1, 	&\textrm{\(b\leq \omega_{2b}(k)<2b\)}
\end{array}
\right. 
\end{align*}
And for  $b<k\leq 2 b $
\begin{align*}
\omega_{2b+1}(k)= \left\{
\begin{array}{ll}
\omega_{2b} (k-1), &\textrm{\(0\leq \omega_{2b}(k-1)<b\)}\\
\omega_{2b} (k-1)+1, &\textrm{\(b\leq \omega_{2b} (k-1)<2b\)}.
\end{array}
\right. 
\end{align*}
Finally, set $\omega_{2b+1}(b)=b$.

As an illustration let $b_n=2^n-1$, such that $b_2=3$, $b_3=7$, $b_4=15$ and $b_5=31$.
Then 
{\small
\begin{gather*}
\sigma_3=(0,1,2), \quad \sigma_7=(0,4,1,3,5,2,6),\\
\sigma_{15}=(0,8,4,12,1,9,3,7,11,5,13,2,10,6,14),\\
\sigma_{31}=(0,16,8,24,4,20,12,28,1,17,9,25,3,19,7,15,23,11,27,5,21,13,29,2,18,10,26,6,22,14,30).
\end{gather*}
}
Moreover, note that for bases of the form $b=2^n$, the algorithm returns permutations that correspond to the first $2^n$ points of the classical van der Corput sequence $\mathcal{S}_2$. 

\subsection{A conjecture}
Faure conjectures that Theorem~\ref{thm92} also holds with the stronger bound $1/(2 \log 2)$ and that this bound is sharp. He partitions the set of integers into intervals $B_n=[2^{n-1},2^n-1]$, sets $d_b^{\omega} =\max_{x \in \RR} \ \psi_b^{\omega}(x)$ and conjectures that
$$ \max_{b \in B_n} \frac{d_b^{\omega}}{\log b} = \frac{d_{b_n}^{\omega}}{\log b_n}, $$
for $b_n = 2^n - 1$. 
In particular he conjectures, based on calculations for $2\leq b \leq 255$, that
\begin{align*} d_{b_n}^{\omega}= \left\{
\begin{array}{ll}
\frac{n}{2} - \frac{1}{3}, 						&\textrm{if $n$ is even,}\\[4pt]
\frac{n}{2} - \frac{1}{3}- \frac{1}{6 b_n}, 	&\textrm{if $n$ is odd}.
\end{array}
\right. 
\end{align*}
Hence, $\frac{d_{b_n}^{\omega}}{\log b_n} \leq \frac{1}{2 \log 2}$ with $\underset{n \rightarrow \infty}\lim \frac{d_{b_n}^{\omega}}{\log b_n} = \frac{1}{2 \log 2}$.
We have carefully investigated this conjecture and slightly sharpen it in the following. Moreover, we will outline in Section~\ref{sec:conclusions} why we believe that this conjecture is important and difficult and why it deserves further attention in the future.
We have extended the computations of Faure and verified his conjecture up to base 1023; see Figure~\ref{fig:conj}. 
\begin{figure}[hbt]
 \centering
 \includegraphics[width=.45\textwidth]{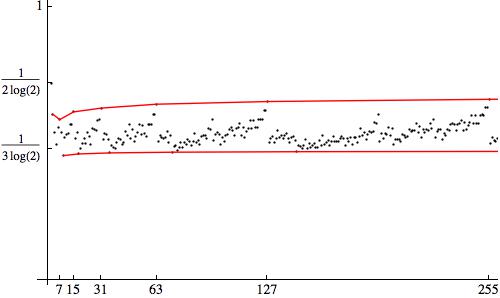} \,\,\,\,
\includegraphics[width=.45\textwidth]{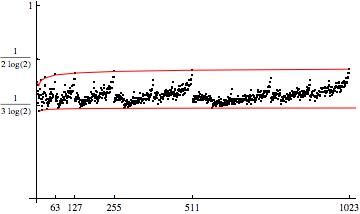}
\caption{\small Plot of $(b, d_b^{\omega}/\log b)$ for $2\leq b \leq 2^{10}-1$ and $\omega$ as defined by Faure. The two red lines connect all minimal resp. maximal values of the intervals $B_n$.}
\label{fig:conj}
\end{figure}
These computations suggest the following sharper version:

\begin{conjecture} \label{conj1}
Let $\omega=\omega_b$ be the permutation obtained from the algorithm of Faure and let $B_n$ and $b_n$ be as above. Then we conjecture for all $n\geq4$,
\begin{equation*}
\min_{b \in B_n} \  \frac{d_b^{\omega}}{\log b} = \frac{ d_{b_{min} }^{\omega} } {\log (b_{min})} \leq \frac{d_b^{\omega}}{\log b} \leq \frac{d_{b_n}^{\omega}}{\log b_n} = \underset{ b \in B_n} \max \ \frac{d_b^{\omega}}{\log b},
\end{equation*}
in which $b \in B_n$ and $b_{min}=2^{n-1} + 2^{n-4} = 9 \cdot 2^{n-4}. $
\end{conjecture}
Interestingly, the minimum is not obtained in bases of the form $2^n$ corresponding to the original van der Corput sequence. We obtain the following result for $\sigma=\omega_{b_{min}}$:
\begin{theorem} \label{thm3}
For every $b_m=9 \cdot 2^m$, $m \geq 0$, there exists a permutation $\sigma$ such that $s(\Sq_{b_m}^{\sigma}) \leq 1/ (3 \log 2) = s(\Sq_2^{id})$.
\end{theorem}

\begin{proof}
Set $I_k:=[k/(2b),(k+1)/(2b)]$, with $k \in \{0,\ldots, 2b-1 \}$. Following \cite{cf93,fa81}, we call the interval $I_k$ \emph{dominated}, if there exists a set $\mathcal{N}$ of integers with $k \notin \mathcal{N}$ such that $\psi_{b}^{\sigma}(x)\leq \underset{j\in\mathcal{N}}{\max\,\,} \psi_{b}^{\sigma}((x+(j-k))/(2b) )$, for all $x \in I_k$. Otherwise the interval is called \emph{dominant}. 

We prove the theorem in two steps. First, we investigate $\psi_9^{\omega}$ and determine its dominant intervals. Second, we study $\psi_{18}^{\omega}$ and deduce a general formula for the dominant intervals of $\psi_{b_m}^{\omega}$, for $b_m=9\cdot 2^m$.

In Table~\ref{tab1} we tabulated the piecewise linear functions $\varphi_{9,h}^{\omega}$ for $0\leq h \leq 8$. From Figure~\ref{fig:psi} it is easy to see that $[5/18,7/18]$ is dominant for $\psi_9^{\omega}$.
Furthermore, note that $\psi_2^{id}(x)=x$, for $x \in [0,1/2]$ and $\psi_2^{id}(x)=-x+1$, for $x \in [1/2,1]$.

\begin{table}[hbt]
 \centering \small
 \begin{tabular}{|c|rrrrr|}
\hline
$\varphi_{9,h}(x) $ 	&	$x\in [0,1/9]$	&	$[1/9,2/9]$	&	$[2/9,3/9]$	&	$[3/9,4/9]$	&	$[4/9,1/2]$ \\
\hline
h=0 	&	0	&	0	&	0	&	0	&	0	\\
1 	&	$8x$ & 	$-x+1$ & 		$-x+1$ & 		$-x+1$ & 		$-x+1$ \\
2 	&  	$7x$ & 	$-2x +1$	& 	$-2x +1$	& 	$-2x +1$	& 	$-2x +1$ \\
3 	& 	$6x$ & 	$-3x +1$ & 	$6x-1$ & 		$-3x +2$ & 	$-3x+2$	\\
4 	& 	$5x$ &  	$-4x +1$ & 	$5x-1$ & 		$-4x +2$ & 	$-3x+2$	\\
5 	& 	$4x$ &  	$-5x +1$ & 	$4x-1$ & 		$-5x +2$ & 	$4x-2$	\\
6 	& 	$3x$ &  	$3x$ & 		$3x$ & 		$-6x +3$ & 	$3x-1$	\\
7 	& 	$2x$ &  	$2x$ & 		$2x$ & 		$-7x +3$ & 	$2x-1$	\\
8 	& 	$x$ &		$x$ &			$x$ &			$x$ &			$x$ \\
\hline
 \end{tabular}
 \caption{\small The piecewise linear functions $\varphi_{9,h}^{\omega}$ on $[0,1/2]$.}
 \label{tab1}
\end{table}

\begin{table}[hbt]
 \centering \small
 \begin{tabular}{|c|cccccccc|}
\hline
		&	$[0,1/9]$	& $[1/9,1/5]$		& $[1/5,2/9]$	&	$ [2/9, 1/4]$	&	$[1/4, 3/9]$	&	$[3/9, 2/5]$	&	$[2/5, 4/9]$ & $[4/9,1/2]$\\
\hline
$\psi_{9}^{\omega}(x)$	& $8x$&	$-x+1$ &	$4x$	& $-5x+2$	&	$3x$ 	& $-3x+2$	&	$2x$ &$-7x+4$\\
\hline
$(\max,\min)$&	(1,0)	&	(1,0)	&(1,5)	&	(1,5)		&	(6,0)	&	(3,0)		&	(3,5)	&(3,5)	\\	
\hline
\end{tabular}
 \caption{\small The function $\psi_{9}^{\omega}$ on $[0,1/2]$ together with the indices $h$ of the corresponding functions $\varphi_{9,h}^{\omega}$.}
 \label{tab2}
\end{table}

For $x \in [5/18,6/18]$ we have $\psi_9^{\omega}(x)=3x$, whereas for $x \in [6/18,7/18]$ we obtain $\psi_9^{\omega}(x)=-3x+2$. 
Recall that $\omega_{18}=  id_2 \cdot \omega_{9}$ and hence Lemma~\ref{intr}
$$\psi_{18}^{\omega}(x)=\psi_{2\cdot 9}^{id_2 \cdot \sigma_{9}}(x) = \psi_{2}^{id_2}(9 x) + \psi_{9}^{\omega_9}(x).$$
Intuitively, we shift and squeeze $\psi_2^{id}(x)$ appropriately and add it to $\psi_9^{\omega}(x)$ in every interval of the form $[k/9, (k+1)/9]$. 
Therefore, we obtain two identical dominant intervals for $\psi_{18}^{\omega}$, namely $[5/18,6/18]$ and $[6/18,7/18]$ and we see that $\psi_{18}^{\omega}$ obtains its maxima at $5/18$ and $7/18$.

In general, we get
$$\psi_{b_m}^{\omega}(x)=\psi_{9}^{\omega}(x) + \sum_{i=1}^m \psi_{2}^{id}(9^i x),$$
which obtains its maximum at
\begin{equation*}
x_m=\frac{3}{9}+\sum_{i=1}^{m} \frac{(-1)^i }{9 \cdot 2^i},
\end{equation*}  
for $m>0$. If $m$ is even, the dominant interval is $J_m=[x_m-1/b_m, x_m]$ and if $m$ is odd, we get $J_m=[x_m, x_m+1/b_m]$, and
\begin{equation*}
\psi_{b_m}^{\omega}(x)=3x + \sum_{i=1}^m (-1)^i (b_{i-1} x - b_{i-1} \ x_{i-1}) = \frac{1}{9} (8+3m+ (-2)^m (-8+27 x) ),
\end{equation*}
for $m>0$ and $x\in J_m$.
Finally,  
$\psi_{b_m}^{\omega}(x_m)=\frac{m+3}{3}$,
and hence,
\begin{equation*}
s(\Sq_{b_m}^{\omega_m}) \leq \frac{m+3}{3 \log b_m} =\frac{m+3}{3 \log (9 \cdot 2^m)} \leq \frac{1}{3 \log 2},
\end{equation*}
as claimed.
\end{proof}

\begin{remark}
We also have a proof confirming the predicted values of $d_{b_n}^{\omega}$. However, this proof is long and technical so we omit it here. The main challenge is anyway to show that these two permutations give indeed the minimum and maximum in each interval $B_n$.
\end{remark}

To conclude this section, we believe that the algorithm of Faure can be interpreted in a vague sense as transferring the structure of the original van der Corput sequence in base 2 to arbitrary integer bases. The numerical results show that the discrepancy of sequences generated by a permutation $\omega$ is always close to the discrepancy of $\mathcal{S}_2^{id}$.

\section{Two families of permutations}
\label{sec:families}

The algorithm of Faure is the main motivation for the results of Pausinger \& Topuzo\u{g}lu \cite{patop18} presented (among other things) in this section. One disadvantage of Faure's algorithm is that we only get one permutation in a given base, and constructing this permutation requires the construction of permutations in smaller bases. In \cite{patop18} the authors aim to give discrepancy bounds for sequences generated from structurally similar permutations in a given (prime) base $p$.
The advantage of restricting to prime bases $p$ is that finite fields $\FF_p$ of $p$ elements are polynomially complete. This means that any self map, and in particular any permutation of $\FF_p$, can be expressed as a polynomial over $\FF_p$. Therefore, we consider in the following permutation polynomials in $\F_p[x]$, where we identify $\F_p$ with $\{0,1, \ldots, p-1\}$.
In this section we are mainly interested in two essentially different families of permutations; i.e. affine permutations and fractional affine permutations. The main idea is to describe structurally similar permutations with very few parameters. Ideally these parameters connect the construction of the permutation with a discrepancy estimate of the resulting sequence and, therefore, provide some insight into the structure of good and weak permutations.

\subsection{Affine permutations}
For fixed but arbitrary $a_0 \in \F_p \setminus \{ 0\}=\F_p^{\ast}$ and $a_1 \in \F_p$ we call the permutation $\sigma = \sigma_{a_0,a_1}$ \emph{affine} if
$$ \sigma_{a_0,a_1} (x) = a_0 x + a_1, $$
and we denote the family of affine permutations in base $p$ with $\Fam_{p}^{(a)}$.
Note that affine permutations are also known as \emph{linear digit scramblings}. This name goes back to a paper of Matou\v sek \cite{mat} and is also discussed in \cite{survey1}. Our notation is motivated by the underlying geometric interpretation and should highlight its algebraic relation to the second family.

Interestingly, we can find an upper bound for $s(\Sq_p^{\sigma})$ for affine permutations $\sigma=\sigma_{a_0,a_1}$ in terms of the parameter $a_0$. We refer to the book of Khinchin \cite{khinchin} for an introduction to continued fractions. Using the standard notation we denote the finite continued fraction expansion of the rational number $\alpha \in [0,1)$ by
$$ \alpha=[0,\alpha_1,\alpha_2, \ldots, \alpha_m]. $$

\begin{theorem}[Pausinger \& Topuzo\u{g}lu \cite{patop18}] \label{thm:affine}
For a prime $p$, let $a_0 \in \F_p^{\ast}$, $a_1 \in \FF_p$ and $\sigma = \sigma_{a_0, a_1} \in \Fam_{p}^{(a)}$. Let $a_0/p=[0,\alpha_1,\alpha_2, \ldots, \alpha_m]$ and set $\alpha_{\max}=\max_{1\leq i \leq m} \alpha_i$. Then, for all $N \in \NN$,
\begin{equation*}
D_N( \Sq_p^{\sigma} ) \leq \frac{\alpha_{\max} + 1}{\log (\alpha_{\max}+1)}  \log (N+1) \ \ \ \text{and} \ \ \ s(\Sq_p^{\sigma}) \leq \frac{\alpha_{\max}+1}{\log (\alpha_{\max}+1)}.
\end{equation*}
\end{theorem}
The main idea of the proof is to combine a result of Niederreiter with the asymptotic method of Faure.
Niederreiter proved an upper bound for the discrepancy of the first $N$ points of any $(\{n \alpha\})$ sequence in terms of $\alpha_{\max}$ provided $\alpha$ has bounded partial quotients; \cite[Chapter~2, Theorem~3.4]{KN74}.
It is observed that the first $N=p$ points of a van der Corput sequence generated from an affine permutation are similarly distributed as the first $N=p$ points of an $(n\alpha)$-sequence if the continued fraction of $a_0/p$ coincides with the initial segment of the expansion of $\alpha$.
The bound of Niederreiter for $1\leq N \leq p$ is used to get upper bounds for the values of Faure's $\psi_b^{\sigma}$-functions.
These upper bounds then suffice to get the stated bounds on the discrepancy constants.

Since $id \in \Fam_{p}^{(a)}$ (for $a_0=1, a_1=0$) it is clear that the number $\alpha_{\max}$ may depend on $p$; i.e. the continued fraction expansion of $1/p$ contains $p$. One interesting problem is therefore to determine, in case there are any, which values of $a_0$ and $p$ guarantee an absolute bound for $\alpha_{\max}$, i.e., a bound, which is independent of $p$ similar to Faure's Theorem~\ref{thm92}. There is a close relation between this problem and the well-known \emph{conjecture of Zaremba} \cite{zar72}. Indeed, recent progress on this conjecture \cite{BoKo13, huang} shows the existence of an infinite set $\mathcal{N}$ of primes such that for each $p\in \mathcal{N}$, there exists an $a_0$ with $\alpha_{\max} \leq 5$.

\begin{remark} \label{rem1}
Since the current results on the conjecture of Zaremba are only for an infinite subset of primes and are non-constructive, one may wonder if Theorem~\ref{thm:affine} is applicable at all.
The bound in Theorem~\ref{thm:affine} is favorable only when $a_0$ and $p$ are chosen such that $\alpha_{\max}$ is small. We call such a parameter $a_0$ a good multiplier in base $p$.  In fact, for practical purposes one can easily obtain good multipliers by looking at the continued fraction expansions of $a_0/p$. 
We refer to \cite[Table~1]{patop18} for a list of multipliers for $11 \leq p \leq 151$, with $\alpha_{\max}(a_0/p)\leq 3$.
\end{remark}

\begin{remark} \label{rem:multFaure}
A generalized Halton sequence is a multi-dimensional sequence, whose $i$-th coordinate is a permuted van der Corput sequence; see also Section~\ref{sec:context}. Faure \cite{fa06} and Faure~\&~Lemieux \cite{FaLe08} derived in their numerical experiments various selection criteria for good (and weak) multipliers $a_0$ (or $f$ in the notation of these papers) for the generation of Halton sequences with small discrepancy.
Theorem~\ref{thm:affine} (and related results from \cite{pa12}) explain some of the observations and suggestions made in \cite{fa06, FaLe08}.
\end{remark}

Let $\mathfrak{F}_n$ denote the $n$th Fibonacci number. It is well known that 
$$\lim_{n\rightarrow \infty} \mathfrak{F}_{n-1} / \mathfrak{F}_{n}= (\sqrt{5}-1)/2$$ 
and the continued fraction expansion of $\mathfrak{F}_{n-1} / \mathfrak{F}_{n}$ contains only 1s (apart from a leading 0). The bound in Theorem~\ref{thm:affine} as well as the main idea of its proof indicate that affine permutations in base $\mathfrak{F}_n$ with multiplier $\mathfrak{F}_{n-1}$ may generate sequences with smallest asymptotic constants within the family of affine permutations. This is also in accordance with the widely held belief that $\alpha=(\sqrt{5}-1)/2$ generates the $(\{n \alpha \})$ sequences with the smallest asymptotic constant. We investigated these special permutations numerically and summarise our results in the following conjecture.

\begin{conjecture} \label{conj2}
Let $p_n=\mathfrak{F}_n$ and $\mu_n=\sigma_{p_{n-1},0}$. Moreover, define the sequence $(z(n))_{n\geq 1}$ as $ z(n):=z(n-1)+z(n-3)+z(n-4), $ with $z(1)=1, z(2)=1, z(3)=1, z(4)=2$. We conjecture that
$$ \max_{x\in [0,1]} \psi_{p_n}^{\mu_n} (x) = \psi_{p_n}^{\mu_n} \left( \frac{z(n-2)}{p_n}\right)= \psi_{p_n}^{\mu_n} \left( \frac{z(n-1)}{p_n}\right).$$
Furthermore, the dominant interval is $\left[ \frac{z(n-2)}{p_n}, \frac{z(n-2)+1}{p_n} \right]$ and
$$ \alpha_{p_n}^{\mu_n} = \lim_{m \rightarrow \infty} \frac{1}{m} \sum_{j=1}^m \psi_{p_n}^{\mu_n} (\hat{x}_n b^j), $$
with $\hat{x}_n = \sum_{m=1}^{\infty} \frac{z(n-2)}{b^m}$.
\end{conjecture}

Based on this conjecture we calculated bounds for $s(\mathcal{S}_{p_n}^{\mu_n})$ which we collect in Table~\ref{table:fib} and Table~\ref{table:fib2}. These results confirm the intuition that linear permutations generate van der Corput sequences whose distribution behavior is similar to what can be achieved with $(\{n\alpha\})$ sequences. Theorem~\ref{thm:affine} links $(\{n\alpha\})$-sequences to generalised van der Corput sequences.
Comparing the best known asymptotic constants due to Faure and Ostromoukhov with the lower bounds in Table~\ref{table:fib} and Table~\ref{table:fib2} raises once again the question: \emph{What is the hidden structure that can not be captured by affine permutations resp. $(\{n\alpha\})$-sequences?}

\begin{table}[h]
\begin{center}
\begin{tabular}{|c|c|c|}
\hline
$n$ & $\mathfrak{F}_n$  &\\
\hline
8 & 21 & $0.4269 \leq s(\mathcal{S}_{21}^{\mu}) < 0.4693$	\\
9 &34 &$0.4382 \leq s(\mathcal{S}_{34}^{\mu}) < 0.4588 $	\\
10 &55 &	$0.4159 \leq s(\mathcal{S}_{55}^{\mu}) < 0.4538$ \\
\hline
11 &89&	$0.4050 \leq s(\mathcal{S}_{89}^{\mu}) < 0.4506$ \\
12 &144&	$0.4221 \leq s(\mathcal{S}_{144}^{\mu}) < 0.4472$\\
13 &233&	$0.4301 \leq s(\mathcal{S}_{233}^{\mu}) < 0.4441$\\
\hline
14 &377&		$0.4160 \leq s(\mathcal{S}_{377}^{\mu})<0.4418$\\
15 &610&		$0.4086 \leq s(\mathcal{S}_{610}^{\mu})< 0.43991$\\
16 &987&		$0.4202 \leq s(\mathcal{S}_{987}^{\mu})<0.4383$\\
\hline
\end{tabular}
\end{center}
\caption{Bounds on asymptotic constants of $\mathcal{S}_{p_n}^{\mu_n}$ generated from Fibonacci-linear permutations with $p_n=\mathfrak{F}_n$ being the $n$th Fibonacci number and $\mu_n$ being the affine permutation with multiplier $p_{n-1}$. The lower bounds are calculated via Conjecture~\ref{conj2}, the upper bounds are calculated from the maximum of the corresponding $\psi$-functions.}
\label{table:fib}
\end{table}

\begin{table}[h]
\begin{center}
\begin{tabular}{|c|c|c|}
\hline
$n$ & $\mathfrak{F}_n$  &\\
\hline
17 &1597 &	$0.4263 \leq s(\mathcal{S}_{\mathfrak{F}_n}^{\mu})$\\
18 &2584 &		$0.4159 \leq s(\mathcal{S}_{\mathfrak{F}_n}^{\mu})$\\
19 &4181 &		$0.4102 \leq s(\mathcal{S}_{\mathfrak{F}_n}^{\mu})$\\
\hline
20 &6765& 		$0.4192 \leq s(\mathcal{S}_{\mathfrak{F}_n}^{\mu})$\\
21 &10946& 		$0.4241 \leq s(\mathcal{S}_{\mathfrak{F}_n}^{\mu})$\\
22 &17711& 		$0.4159 \leq s(\mathcal{S}_{\mathfrak{F}_n}^{\mu})$\\
\hline
23 &28657& 		$0.4112 \leq s(\mathcal{S}_{\mathfrak{F}_n}^{\mu})$\\
24 &46368& 		$0.4185 \leq s(\mathcal{S}_{\mathfrak{F}_n}^{\mu})$\\
25 &75025& 		$0.4226 \leq s(\mathcal{S}_{\mathfrak{F}_n}^{\mu})$\\
\hline
26 &121393& 		$0.4158 \leq s(\mathcal{S}_{\mathfrak{F}_n}^{\mu})$\\
27 &196418& 		$0.4119 \leq s(\mathcal{S}_{\mathfrak{F}_n}^{\mu})$\\
28 &317811& 		$0.4181 \leq s(\mathcal{S}_{\mathfrak{F}_n}^{\mu})$\\
\hline
29 &514229& 		$0.4216 \leq  s(\mathcal{S}_{\mathfrak{F}_n}^{\mu})$\\
30 &832040& 		$0.4158 \leq s(\mathcal{S}_{\mathfrak{F}_n}^{\mu})$\\
\hline
\end{tabular}
\end{center}
\caption{Bounds on asymptotic constants for Fibonacci-linear, $\mathcal{S}_{p_n}^{\mu_n}$, with $p_n=\mathfrak{F}_n$ being the $n$th Fibonacci number and $\mu_n$ being the affine permutation with multiplier $p_{n-1}$. The lower bounds are calculated via Conjecture~\ref{conj2}.}
\label{table:fib2}
\end{table}

\subsection{Fractional affine permutations}
For $a_0 \in \F_p^{\ast}$ and $a_1, a_2 \in \F_p$ we call the permutation $\pi = \pi_{a_0,a_1,a_2}$ \emph{fractional affine} if 
$$ \pi_{a_0,a_1,a_2} (x) = (a_0 x +a_1)^{p-2} + a_2, $$
and we denote the family of fractional affine permutations with $\Fam_{p}^{(f)}$. This family does not contain the identity; in fact $\Fam_{p}^{(a)}$ and $\Fam_{p}^{(f)}$ are always disjoint. Interestingly, it turns out that permutations in $\Fam_{p}^{(f)}$ define sequences all of which are better distributed than the classical van der Corput sequence in the same base; compare \eqref{id} and \eqref{upperBound2}. 

\begin{theorem}[Pausinger \& Topuzoglu \cite{patop18}] \label{thm:fractional}
Let $a_0 \in \F_p^{\ast}$ and $a_1, a_2 \in \FF_p$ and let $\pi=\pi_{a_0,a_1,a_2} \in \Fam_{p}^{(f)}$. Then,  
\begin{equation} \label{upperBound2}
s(\mathcal{S}_p^{\pi}) <  s(\mathcal{S}_p^{id}).
\end{equation}
\end{theorem}

\begin{remark} \label{rem2}
Even if the set of permutations $\Fam_{p}^{(f)}$ is much larger than $\Fam_{p}^{(a)}$, the range of values for $s(S_p^{\pi})$ is smaller. 
While $\Fam_{p}^{(a)}$ contains permutations generating sequences with very small as well as largest possible discrepancy (in the context of permuted van der Corput sequences), the permutations in $\Fam_{p}^{(f)}$ avoid this extremal behavior; see \cite[Table~2]{patop18} for numerical results. 
\end{remark}


To get an idea of the proof of this result 
let $\sigma \in \Sy_b$. Faure showed \cite[Corollaire 3]{fa81} that
\begin{equation} \label{chi}
\psi_b^{\sigma}\left(  \frac{k}{b} \right) \leq  k \left( 1-\frac{k}{b} \right),
\end{equation}
for $0\leq k \leq b-1$. Since $D_k(S_b^{\sigma})= \psi_b^{\sigma}(k/b)$ for $1\leq k \leq b-1$, we obtain equality in \eqref{chi} if and only if all $k$ points lie in an interval of length $k/b$, since
$ k - \frac{k^2}{b} = k \left(1 - \frac{k}{b} \right). $ 
The identity permutation satisfies this for every $k$ from which Faure obtains
$$\underset{1\leq k \leq b} \max \, \psi_b^{id}(k/b) = \psi_b^{id}\left (\frac{\lfloor b/2 \rfloor } {b} \right) \ \ \ \text{ and } \ \ \ \psi_b^{\sigma}(x) \leq  \psi_b^{id}(x),$$
for all $x \in [0,1]$ and all $\sigma \in \Sy_b$. In addition,
$$\psi_b^{id}(k/b) = \psi_b^{\sigma}(k/b) $$
whenever $V_k^{\sigma}=\{ 1, \ldots, k\} \oplus a$, where $a \in \NN$ is constant and $\oplus$ denotes addition modulo $b$.
Moreover, Faure finds for odd $b$ \cite[Th\'{e}or\`{e}me 6]{fa81} that 
$$\alpha_b^{id} = \lim_{n\rightarrow \infty} \alpha_{b,n}^{id} \ \ \ \text{ with } \ \ \
 \alpha_{b,n}^{id} = \frac{1}{n}\sum_{j=1}^{n} \psi_b^{id} \left( \frac{\tilde{x}_n}{ b^j} \right), $$
in which 
$$\tilde{x}_n = \sum_{j=1}^{n} \frac{b-1}{2} b^{j-1}.$$


It turns out \cite[Lemma~4.1]{patop18} that fractional affine permutations never map the set $\{0,1,\ldots, (p-1)/2 \}$ to a set of the form $V_{(p-1)/2}^{id} \oplus a$, for an $a \in \FF_p$. Consequently,
\begin{equation*}
\underset{1\leq k \leq p} \max \psi_b^{\pi}(k/p) < \underset{1\leq k \leq p} \max \psi_p^{id}(k/p),
\end{equation*}
for all $\pi = \pi_{a_0,a_1,a_2} \in \Fam_{p}^{(f)}$ from which the result follows.

The observations and calculations mentioned in Remark~\ref{rem2} suggest the following conjecture.
\begin{conjecture}\label{conj:fractional}
There exists an increasing function $\kappa: \NN \rightarrow \RR$ with $\kappa(p) < s(\mathcal{S}_p^{\pi})$ for all $\pi \in \Fam_{p}^{(f)}$ and $\lim_{n \rightarrow \infty} \kappa(n) = \infty$. In other words, there is no infinite subset of fractional affine permutations such that the corresponding asymptotic constants can be bounded by an absolute constant -- in contrast to the infinite sets of permutations in Theorem~\ref{thm92} and Theorem~\ref{thm:affine}.
\end{conjecture}

We can even go one step further and extend the family of affine permutations. First, we consider fractional linear transformations 
$$R_1(x)= \frac{\alpha_2 x + \beta_2}{\alpha_1x + \beta_1},~\alpha_2 \beta_1-\beta_2 \alpha_1\neq 0,$$ and the permutations of $\Fp$, defined as $\bar{\pi}(x)=R_1(x)$ for $x \in \Fp \setminus \{-\beta_1/\alpha_1\}$, and
$\bar{\pi}(-\beta_1/\alpha_1)= -\alpha_2/\alpha_1.$ Clearly $\bar{\pi}(x)$ can be expressed as 
\begin{equation*} 
\bar{\pi}(x)=\pi_{a_0,a_1,a_2}(x)=( a_0 x +a_1)^{p-2} + a_2,
\end{equation*} 
where 
$a_0 \neq 0, \alpha_1 = a_0, \beta_1=a_1, \alpha_2 = a_0a_2, \beta_2=a_1a_2+1.$ 
Similarly we consider permutations 
\begin{equation*}
\tau(x)=\tau_{A_0, A_1, A_2, A_3}(x)=(( A_0 x +A_1)^{p-2} + A_2)^{p-2}+A_3
\end{equation*}
for $A_0, A_2 \in \FF_p^{\ast}$ and $A_1, A_3 \in \FF_p$, and the 
fractional transformations 
$$R_2(x)= \frac{\alpha_3 x + \beta_3}{\alpha_2x + \beta_2},$$ 
where $\alpha_2, \beta_2$ are as above, and
$\alpha_3 = A_0(A_2 A_3+1),~ \beta_3=A_1(A_2 A_3+1)+a_3$. We note that  
$\tau(x)=R_2(x)$ for $x \in \Fp \setminus \{X_1, X_2 \}$, with
$X_1=-A_1/A_0$, $X_2=-(A_1 A_2+1)/(A_0 A_2)$, and set
$\tau(X_1)=\alpha_3/\alpha_2$, $\tau(X_2)=R_2(X_1)$.

Interestingly it was observed in \cite{patop18} that for every permutation $\pi$ there exists a permutation $\tau$ such that $\pi(x)=\tau(x)$ for all $x \in \FF_p \setminus \{ X_1, X_2\}$; i.e. we can attach $p-1$ permutations $\tau$ to each permutation $\pi$.

\begin{theorem}[Pausinger \& Topuzoglu \cite{patop18}] \label{lem:relation}
Fix $a_0, a_1, a_2 \in \FF_p$, with $a_0, a_2 \neq 0$ and set (in $\FF_p$)
$$A_0=-a_0 a_2^2, \ \ \ A_1=-a_1 a_2^2 - a_2, \ \ \ A_2=1/a_2.$$ 
Consider $\pi=\pi_{a_0,a_1,a_2} \in \Fam_{p}^{(f)}$ and $\tau=\tau_{A_0,A_1,A_2,0}$. Then $\pi(x)=\tau(x)$
for all $x \in \F_p$, except for 
$$X_1 = -(a_1 a_2 +1)/a_0 a_2, \ \ \  X_2 =-a_1/a_0,$$ 
for which $\pi(X_1)=\tau(X_2)$ and $\pi(X_2)=\tau(X_1)$.
\end{theorem}
It turns out that all permutations in a fixed base $p$ of the form $\tau_{A_0,A_1,A_2,A_3}$ generate sequences whose discrepancy is very similar to the corresponding permutation $\pi_{a_0,a_1,a_2}$. In particular, extremal discrepancy behaviour is again omitted (see Remark~\ref{rem2}). 
\begin{conjecture}
We believe Theorem~\ref{thm:fractional} and Conjecture~\ref{conj:fractional} also hold for the larger but structurally similar set of permutations of the form $\tau_{A_0,A_1,A_2,A_3}$.
\end{conjecture}

\subsection{Directions for future work}
Theorem~\ref{thm:affine} and \ref{thm:fractional} offer two assets to the practitioner. Firstly, they provide a criterion based on continued fractions to choose a provably \emph{good multiplier} for linear digit scrambling in prime base $p$. Secondly, they show that picking any permutation from $\Fam_{p}^{(a)}$ ensures to avoid extremal discrepancy behavior of the resulting sequence - independent of the particular choice of parameters. That is, any choice of parameters $a_0, a_1, a_2$ gives a sequence that is better than the worst and worse than the best generalised van der Corput sequences in base $p$.

Every permutation of $\F_p$ can be represented by a polynomial 
\begin{equation*}
P_n(x)=( \ldots ( ( a_0 x +a_1)^{p-2} + a_2)^{p-2} \ldots +a_n)^{p-2} +a_{n+1},~n\geq0 
\end{equation*}
for $a_0a_2 \cdots a_n \neq 0$, with an associated fractional transformation 
$$R_n(x)= \frac{\alpha_{n+1} x + \beta_{n+1}}{\alpha_nx + \beta_n},$$ 
where $\alpha_i, \beta_i, ~i \geq2$
can be described recursively.
This is due to a well-known result
of Carlitz \cite{ca53}, and leads to the concept of the \textit{Carlitz rank} of permutations. For details
we refer to \cite{to13} and the references therein.
The Carlitz rank is a particular measure of the complexity of a permutation. The results of this section can be seen as a study of permutations of small Carlitz rank, i.e. Carlitz rank $0,1,2$. Thus, this is a first step towards a systematic study of the distribution properties of permutations of fixed Carlitz rank $n$.
Numerical investigations of permutations of Carlitz rank 3 show that certain subsets improve the smallest values obtained for affine permutations.

\begin{problem}
Study permutations of larger Carlitz rank and characterise the structure of those permutations that improve the best results for affine permutations. Prove results similar to Theorem~\ref{thm:affine} and \ref{thm:fractional}.
\end{problem}

\begin{problem}
Study the permutation of Ostromoukhov in base $84$. Is there a compact way to describe Ostromoukhov's permutation as a polynomial resp. with few parameters?
\end{problem}

\section{Sequences of permutations}
\label{sec:sequences}

A second approach to improving the asymptotic constants of classical van der Corput sequences is to study particular sequences of permutations. In this context the Swapping Lemma plays a crucial role. The result of Faure \cite[Th\'{e}or\`{e}me 3]{fa81} for the star discrepancy presented in Section~\ref{sec:asymp} is a first success in this direction. Since $\psi_b^{id,-}=0$ it is possible to calculate the asymptotic constants for sequences $\mathcal{S}_b^{\Sigma_A^{id}}$. In fact,
$$ s^*(\mathcal{S}_b^{\Sigma_A^{id}}) = \frac{\alpha_b^{id,+}}{2 \log b} = \left\{
\begin{array}{ll}
\frac{b-1}{8 \log b} & \mbox{ if $b$ is odd,}\\
\frac{b^2}{8 (b+1) \log b}& \mbox{ if $b$ is even. }
\end{array} \right.$$
Thus, this construction improves the constants $s^*(\mathcal{S}_b^{id})$ obtained for the classical van der Corput sequences by a factor of 2 -- which is known to be the best possible reduction; see \cite[Section~2.3]{survey1}. The smallest value is obtained for $b=3$ with $s^*(\mathcal{S}_3^{\Sigma_A^{id}})=0.2275\ldots$ being quite close to the currently smallest known value.

In this context there is another interesting structural result obtained for $b=2$ by Kritzer, Larcher \& Pillichshammer in \cite{KLP07}. They consider general sequences $\Sigma$ of the two permutations  $id_2=(0,1)$ and $\tau_2=(1,0)$ in base 2 and define for $m \in \NN$
$$ S_m(\Sigma) := \max ( \#\{ 0 \leq j \leq m-1: \sigma_j=\tau_2 \}, \#\{ 0 \leq j \leq m-1: \sigma_j=id_2 \} ) $$
and
$$ T_m(\Sigma) := \# \{ 1 \leq j \leq m-1 : \sigma_{j-1}= \tau_2 \ \ \text{ and } \ \ \sigma_j=id_2 \}. $$
It turns out that these two quantities can be used to give precise discrepancy estimates for generalised van der Corput sequences.

\begin{theorem}[Kritzer, Larcher \& Pillichshammer \cite{KLP07}] \label{thm:strucutre}
If $\Sigma=(\sigma_j)_{j \geq 0}$ is a sequence of the two permutations $id_2$ and $\tau_2$, then
$$ \frac{S_m(\Sigma)}{3} + \frac{T_m(\Sigma)}{48} - 4 \leq \max_{1 \leq N \leq 2^m} D_N^*(S_b^{\Sigma}) \leq  \frac{S_m(\Sigma)}{3} + \frac{2 T_m(\Sigma)}{9} + \frac{56}{9}. $$
\end{theorem}

This theorem shows that essentially two properties of $\Sigma$ influence the star discrepancy of generalised van der Corput sequences. These are the number of identity permutations compared to the number of transpositions as well as their distribution. In particular, it shows that it is favourable to have, for each $m$, about the same number of identity permutations and transpositions with as few changes as possible. This result hints that the original choice of Faure from \cite[Th\'{e}or\`{e}me 3]{fa81} is best possible.

It would be very interesting to extend Theorem~\ref{thm:strucutre} beyond base 2. Sequences $\Sigma_A^{\sigma}$ having the desirable structure have been studied in the context of Hammersley point sets as discussed in the next section. However, it seems all these approaches rely on the Swapping Lemma in a crucial way -- what happens in the case of more involved sequences of permutations?

\begin{problem}
Study more general sequences of permutations and extend Theorem~\ref{thm:strucutre} beyond base 2. What happens for example in base 3 or 5? Is $s^*(\mathcal{S}_3^{\Sigma_A^{id}})$ already best possible among all sequences of permutations in base 3? What happens if we use sequences of affine and swapped affine permutations in larger bases?
\end{problem}

\section{Faure's method in a wider context}
\label{sec:context}

The aim of this section is to give an idea how Faure's methods have been used in a wider context over the years. We highlight a number of theorems that would be immediately affected by any improvement of the result of Ostromoukhov. 

\subsection{Generalised Hammersley point sets}
The definition of star and extreme discrepancy for multi-dimensional point sets and sequences is a straightforward generalisation of the one dimensional notion in which the supremum is taken over all axis-aligned sub-rectangles of $[0,1]^d$ instead of all subintervals of $[0,1]$; see \cite{DP, DT, KN74}.
The radical inverse function $S_b(n)$ can be used to construct multidimensional point sets and sequences with low discrepancy. 
Kuipers \& Niederreiter write in the Notes to Section~3 in \cite[Chapter 2]{KN74} that such constructions go back to the seminal paper of Roth \cite{roth54} from 1954 in which he used the radical inverse function $S_2(n)$ to construct a point set in two dimensions. Later Halton \cite{halton} showed that the infinite sequence $\mathcal{H}_{\bf b}=(H_{\bf b}(n))_{n\geq1}$, in which $\mathbf{b}=(b_1,\ldots,b_d)$ is a $d$-tuple of pairwise coprime integers and
$$ H_{\bf b}(n) = (S_{b_1}(n), \ldots, S_{b_d}(n)) $$
is uniformly distributed \cite{halton}. Sequences of the form $\mathcal{H}_{\bf b}$ are nowadays referred to as \emph{Halton sequences}. 
In the context of numerical integration there is the related concept of a \emph{Hammersley point set} \cite{hammersley}. 
The main difference is that the number of points, $N$, is fixed in the case of Hammersley point sets and the last coordinate of the $n$th point is equal to $n/N$ for $1\leq n \leq N$. It is widely believed that both constructions yield examples of point sets and sequences with the optimal order of star discrepancy. However, this question is only settled for one dimensional sequences and two dimensional point sets (see \cite{sch72}) and is still open in larger dimensions.

In \cite{fa08} Faure considered generalised two-dimensional Hammersley point sets, $\mathcal{H}_{b,m}^{\Sigma}$, which are deduced from generalised van der Corput sequences and defined as
$$ \mathcal{H}_{b,m}^{\Sigma} := \left \{ \left( S_b^{\Sigma}(n), \frac{n-1}{b^m} \right) : 1 \leq n \leq b^m \right \}. $$
Faure remarks that in order to match the traditional definitions of arbitrary Hammersley point sets which are $m$-bits, i.e. whose $b$-adic expansions do not exceed $m$ bits, he restricts to infinite sequences of permutations $\Sigma$ such that $\sigma_r(0)=0$ for all $r \geq m$, such that the distribution behaviour of $\mathcal{H}_{b,m}^{\Sigma}$ only depends on the first $m$ permutations 
$$\bm{\sigma}=(\sigma_0, \ldots, \sigma_{m-1}).$$ 
This restriction is indicated by writing $\mathcal{H}_{b,m}^{\bm{\sigma}}$.
The classical Hammersley point set in base $b$ can be obtained if we set $\sigma_r=id$ for all permutations. Faure discusses in the introduction of \cite{fa08} important contributions of Halton \& Zaremba \cite{HalZar}, De Clerck \cite{DeCl} and Larcher \& Pillichshammer \cite{LarPi} to the exact calculation of the (star) discrepancy of particular two-dimensional Hammersley point sets as well as approximative formulas for the leading terms within a small error (usually lower than a small additive constant that does not play a role in the asymptotic analysis of the discrepancy) by himself \cite{fa86}, Kritzer \cite{kri06} and Kritzer, Larcher \& Pillichshammer \cite{KLP07}.

The main contribution of Faure's paper is the following approximate formula \cite[Theorem~1]{fa08} for the star discrepancy of generalised Hammersley point sets.

\begin{theorem}[Faure \cite{fa08}] \label{thm:ham}
For any $m \in \NN$ and any $\bm{\sigma}=(\sigma_0, \ldots, \sigma_{m-1}) \in \Sy_b^m$ we have with some $c_m \in [0,2]$,
$$D_{b^m}^*(\mathcal{H}_{b,m}^{\bm{\sigma}}) = 
\max \left( 
\max_{1 \leq n \leq b^m} \sum_{j=1}^m \psi_b^{\sigma_{j-1},+} \left( \frac{n}{b^j} \right), 
\max_{1 \leq n \leq b^m} \sum_{j=1}^m \psi_b^{\sigma_{j-1},-} \left( \frac{n}{b^j} \right)
\right) + c_m.$$
\end{theorem}

He continues his investigations with a detailed study of particular sequences of permutations $\bm{\sigma}$. In particular, he studies vectors of permutations consisting of $id$ and $\tau_b$.
Following Kritzer who studied such arrangements in base 2 \cite{kri06}, he defines
$$ \bm{i\tau}= (id, \ldots, id, \tau, \ldots, \tau), $$
in which, for even $m$, $m/2$ copies of $id$ are followed by $m/2$ copies of $\tau$ or, in the case of odd $m$, $(m-1)/2$ copies of $id$ are followed by $(m+1)/2$ copies of $\tau$.
Using Theorem~\ref{thm:ham} Faure is able to derive a formula for $\mathcal{H}_{b,m}^{\bm{i \tau}}$ in \cite[Theorem~2]{fa08} and, in particular, derives the following asymptotic results in \cite[Corollary 2]{fa08}:
\begin{align*}
\lim_{m \rightarrow \infty} \frac{D_{b^m}^*(\mathcal{H}_{b,m}^{\bm{i \tau}})}{\log b^m} = \frac{b-1}{8 \log b}, \ \ \ \ \text{if $b$ is odd and}\\
\lim_{m \rightarrow \infty} \frac{D_{b^m}^*(\mathcal{H}_{b,m}^{\bm{i \tau}})}{\log b^m} = \frac{b^2}{8(b+1) \log b}, \ \ \ \ \text{if $b$ is even.}
\end{align*}
This recovers the earlier results of \cite{kri06, KLP07} for $b=2$ and shows that the best constant is obtained for $b=3$ with $1/(4 \log 3) = 0.227\ldots$.
Interestingly, it is shown \cite[Theorem~3]{fa08} that these asymptotic values can not be improved by any other $\bm{\sigma}$ in which each component $\sigma_j$, $1\leq j\leq m$ is either $id$ or $\tau$. However, \cite[Theorem~4]{fa08} shows that the asymptotic values actually \emph{depend} on the distribution of $id$ and $\tau$ in $\bm{\sigma}$.

Furthermore, Faure considers swapping with an arbitrary permutation and considers sequences $\bm{\sigma}$ that are generated from permutations $\sigma$ and $\bar{\sigma}:=\tau \circ \sigma$. Defining $\bm{\sigma \bar{\sigma}}$ analogously to $\bm{i \tau}$, he obtains a formula for $D_{b^m}^*(\mathcal{H}_{b,m}^{\bm{\sigma \bar{\sigma}}})$ in \cite[Theorem~5]{fa08} and gets the asymptotic formula
\begin{equation} \label{eq:ham}
\lim_{m\rightarrow \infty} \frac{D_{b^m}^*(\mathcal{H}_{b,m}^{\bm{\sigma \bar{\sigma}})}}{\log b^m} = \frac{\alpha_b^{\sigma,+}+\alpha_b^{\sigma,-}} {2 \log b}.
\end{equation}
From this he is able to slightly improve the best asymptotic results for $\bm{\sigma}=\bm{i \tau}$ by using those permutations that also give the best possible asymptotic star discrepancy constant for generalised van der Corput sequences.
Ostromoukhow shows in \cite[Theorem~5]{os09} that for his record permutation, $\sigma_{60}=\sigma$, in base 60 $\psi_{60}^{\sigma,-}=0$ and hence $\psi_{60}^{\sigma}=\psi_{60}^{\sigma,+}$. Consequently, $\alpha_{60}^{\sigma,-}=0$ and we get from \eqref{eq:ham}
$$ \lim_{m\rightarrow \infty} \frac{D_{60^m}^*(\mathcal{H}_{60,m}^{\bm{\sigma \bar{\sigma}})}}{\log 60^m} = \frac{\alpha_b^{\sigma,+}} {2 \log 60} = \frac{32209}{17700} \frac{1}{2 \log 60} = 0.222223. $$

Finally, Faure and Pillichshammer also thoroughly studied the $L_p$-discrepancy of generalised Hammersley point sets and derived various interesting results. In particular, they showed in \cite{FaPi09a} that the $L_p$-discrepancy, for finite $p$, of classical Hammersley point sets is not of optimal order according to the known lower bounds for arbitrary $N$-element point sets. Interestingly, this can be overcome when using certain generalised Hammerlsey point sets. In particular there are various papers that provide constructions of such point sets with optimal order of $L_2$-discrepancy; see \cite{FaPi09a, FaPi09b, FaPi11}, \cite{FaEtAl10} with Schmid and Pirsic and \cite{FaEtAl11} with Pirsic as well as \cite[Section~4.1]{survey1} for a more information.

\subsection{Digital $(0,1)$-sequences in prime base}
It can be easily shown that for every $m \in \NN$ and $k\in \{0,1,\ldots, b^m-1\}$ exactly one of $b^m$ consecutive elements of the van der Corput sequence in base $b$ belongs to a $b$-adic elementary interval of the form $[k/b^m,(k+1)/b^m)$. This important property forms the basis for the definition of the much larger set of $(0,1)$-sequences whose definition goes back to Niederreiter \cite{nie92}. The methods of Faure can be extended to this larger set of sequences. In the following we briefly review some of the results in this direction without going into details; we refer to \cite[Section~3.2]{survey1} for a more thorough introduction of the related concepts.

Interestingly, it is known that among \emph{all} $(0,1)$-sequences in base $b$ the classical van der Corput sequence in base $b$ has the worst star discrepancy \cite{fa07, kri05, pil04}. In \cite{fa05} Faure considers a particular subset of digital $(0,1)$-sequences, so called NUT digital $(0,1)$-sequences over $\FF_p$, with $p$ prime, for which he obtains formulas for the discrepancy in the spirit of the results of Section~\ref{sec:Disc}. These results were later extended to arbitrary integer bases resp. to the even larger class of NUT $(0,1)$-sequences over $\ZZ_b$ \cite{FaPi13}.

Finally, and this is the main reason why we gave this brief outlook, it turns out that the asymptotic constants $\alpha_b^{\sigma}$ can be used to formulate best possible lower bounds for the star discrepancy of NUT $(0,1)$-sequences over $\ZZ_b$, which form so far the most general family of van der Corput type sequences for which such a result exists; see \cite[Theorem~36 and its corollaries]{survey1}. In this context, the permutation of Ostromoukhov can again be used to generate a sequence with to date smallest known asymptotic constant.

\section{Why things remain intriguing}
\label{sec:conclusions}



We have given a broad overview of the use of permutations in uniform distribution theory. Henri Faure provided the necessary tools to calculate and compare sequences generated from different permutations. Even the most modern and general constructions of low discrepancy sequences can benefit from good \emph{seeds}, i.e. good generating permutations.
Despite all the efforts and the various great structural results, our main question remains open. \emph{What is the hidden structure behind the best permutations?} We have seen that this structure is not linear-like since $(\{n \alpha \})$ sequences as well as van der Corput sequences generated from affine permutations can be improved using the permutations of Ostromoukhov. So the structure imposed by affine permutations seems to be too rigid.
We have also seen that a clever alternation of $(0,1)$ and $(1,0)$ gives sequences of permutations whose corresponding van der Corput sequence has almost smallest known star discrepancy. However, the results of Faure for Hammersley point sets hint that using Ostromoukhov's permutations in a similar fashion yields once again slightly better results.

Thus the problem remains intriguing! Is there such thing as a best permutation? And if yes, how could we know that we have found it? We believe that the infimum in the definition of $\hat{s}$ and $\hat{s}^*$ is indeed an infimum and that we can always find a permutation in a larger base that slightly improves the so far smallest value. For this reason we find Conjecture~\ref{conj1} very interesting. Because the main difficulty in this conjecture is the fact that we study infinitely many sequences whose asymptotic constants seems to converge to a particular value. We believe that the solution of this conjecture may require an idea that could also be useful in the search for optimal permutations resp. for improving the results of Ostromoukhov. 


\appendix
\section{List of good permutations} \label{appendix}

In this appendix we list the permutations that generate generalised van der Corput sequences with to-date smallest known asymptotic extreme discrepancy, star discrepancy and diaphony constants.
We reformulate the corresponding results in our notation.

\subsection{Extreme discrepancy}
Let 
\begin{align*}
\sigma_{36}=(0,&25,17,7,31,11,20,3,27,13,34,22,5,15,29,9,23,1, \\ & 18,32,8,28,14,4,21,33,12,26,2,19,10,30,6,16,24,35),
\end{align*}
and let
\begin{align*}
\sigma_{84}=(0,& 22, 64,32,50,76,10,38,56,18,72,45,6,28,59,79,41,13,67,25,54, \\ 
& 2,36,70,16,48,81,30,61,8,43,74,20,52,4,34,66,15,46,77,26,11, \\ 
& \hspace{0.5cm} 62, 39,82,57,23,69,33,3,51,19,73,42,7,60,29,80,47,14,65,35,1, \\ 
& \hspace{1cm} 53,24, 68, 12, 40, 78, 58, 27, 5, 44, 71, 17, 55, 37, 83, 21, 49, 75, 9, 31, 63).
\end{align*}

\begin{theorem}[Faure \cite{fa92}, Ostromoukhov \cite{os09}] For $b=36$ and $\sigma_{36}$ we have
$$ s(\mathcal{S}_{36}^{\sigma_{36}}) = \frac{46}{35 \log 36} = 0.3667\ldots ; $$
for $b=84$ and $\sigma_{84}$ we have
$$ s(\mathcal{S}_{84}^{\sigma_{84}}) = \frac{130}{83 \log 84} = 0.3534\ldots . $$
\end{theorem}

\subsection{Star discrepancy}
Let 
\begin{align*}
\sigma_{60}=(0,&15,30,40,2,48,20,35,8,52,23,43,12,26,55,4,32,45,17,37, \\
& 6,50,28,10,57,21,41,13,33,54,1,25,46,18,38,5,49,29,9,58, \\
& \hspace{0.5cm} 22, 42, 14, 34, 53, 3, 27, 47, 16, 36, 7, 51, 19, 44, 31, 11, 56, 24, 39, 59).
\end{align*}

\begin{theorem}[Ostromoukhov \cite{os09}] For $b=60$ and $\sigma_{60}$ we have
$$ s^*(\mathcal{S}_{60}^{\Sigma_{A}^{\sigma_{60}}}) = \frac{32209}{35400 \log 60} = 0.222223\ldots . $$
\end{theorem}

\subsection{Diaphony}
Let
\begin{align*}
\sigma_{19}=(0,11,5,15,9,3,17,7,13,1,12,6,16,2,8,14,4,10,18),
\end{align*}
and let
\begin{align*}
\sigma_{57}=(0,&24,37,8,43,18,52,29,11,48,33,4,21,40, 14, 54, 26, 45, 6, \\
& 35, 16, 50, 31, 2, 20, 39, 10, 47, 27, 55, 13, 42, 23, 3, 32, 51, 17, 36, \\
& \hspace{0.5cm} 7, 46, 28, 56, 15, 41, 22, 5, 34, 49, 9, 25, 53, 38, 12, 30, 1, 19, 44).
\end{align*}

\begin{theorem}[Chaix \& Faure \cite{cf93}, Pausinger \& Schmid \cite{ps10}] For $b=19$ and $\sigma_{19}$ we have
$$ f(\mathcal{S}_{19}^{\sigma_{19}}) = \frac{3826 \pi^2}{27 \cdot 19^2 \log 19} = 1.31574\ldots ; $$
and for $b=57$ and $\sigma_{57}$ we have
$$ f(\mathcal{S}_{57}^{\sigma_{57}}) = \frac{42407 \pi^2}{28 \cdot 57^2 \log 57} = 1.13794\ldots . $$
\end{theorem}


\end{document}